\documentclass[11pt]{article}
\usepackage[T1]{fontenc}
 \usepackage{ae,aecompl}

\usepackage{enumerate}
\usepackage{amsthm}
\usepackage{a4}
\usepackage{amsmath}

\usepackage{amssymb}
\usepackage{graphics}
\usepackage{color}
\usepackage[colorlinks=true,urlcolor=blue,citecolor=blue,linkcolor=blue,bookmarks=true,bookmarksopen=true,pdfstartview=FitH]{hyperref}
\usepackage{vaucanson-g}

\theoremstyle{plain}
\newtheorem{th-def}{Theorem-Definition}[section]
\newtheorem{theo}[th-def]{Theorem}
\newtheorem{lem}[th-def]{Lemma}
\newtheorem{prop}[th-def]{Proposition}
\newtheorem{coro}[th-def]{Corollary}

\theoremstyle{definition}
\newtheorem{defi}[th-def]{Definition}

\theoremstyle{remark}
\newtheorem{rem}{\noindent Remark}[section]
\newtheorem{rems}[rem]{\noindent Remarks}
\newtheorem{ex}{\noindent Example}

\newcommand {\junk}[1]{}

\def\N{\mathbb N}
\def\Z{\mathbb Z}

\def\R{\mathbb R}

\def\P{\mathbb P}
\def\L{\mathbb L}

\def\E{\mathbb E}
\def\A{\mathcal A}
\def\mE{\mathcal E}

\def\rma{{\mathbb R}_{\max}}

\def\mp{max-plus }

\def\G{\mathcal{G}}

\def\sAn{\left(A(n)\right)_{n\in\N}}
\def\sXn{\left(x(n,0)\right)_{n\in\N}}

\def\1{\mathbf{1}}

\title{Cycle time of stochastic max-plus linear systems.}

\author{Glenn MERLET\footnote{This article is based on my work during my PhD at Universit\'e de Rennes~1, as a JSPS postdoctoral fellow at Keio University, and as ATER at Universit\'e Paris-Dauphine. It was also supported by the ANR project MASED (06-JCJC-0069).}\\
LIAFA, CNRS-Universit\'e Paris-Diderot\\
Case 7014\\
F-75205 Paris Cedex 13 \\
E.mail: \url{glenn.merlet@gmail.com}}
\date{}

\begin{document}
\maketitle
\vfill
\begin{abstract}
We analyze the asymptotic behavior of sequences of random variables $\left(x(n)\right)_{n\in\N}$ defined by an initial condition and the induction formula $x_i(n+1)=\max_j\left(A_{ij}(n)+x_j(n)\right)$, where $\sAn$ is a stationary and ergodic sequence of random matrices with entries in \mbox{$\R\cup\{-\infty\}$}.

This type of recursive sequences are frequently used in applied probability as they model many systems as some queueing networks, train and computer networks, and production systems.

We give a necessary  condition for $\left(\frac{1}{n}x(n)\right)_{n\in\N}$ to converge almost-surely, which proves to be sufficient when the $A(n)$ are i.i.d.

Moreover, we construct a new example, in which $\sAn$ is strongly mixing, that condition is satisfied, but $\left(\frac{1}{n}x(n)\right)_{n\in\N}$ does not converge almost-surely.
\end{abstract}
\vfill

\noindent{\bf Keywords:} LLN; law of large numbers ; subadditivity ; Markov chains ; max-plus ; stochastic recursive sequences ; products of random matrices.\\

\noindent{\bf AMS-Classification:} Primary 60F15, 93C65; Secondary 60J10; 90B15; 93D209\\
\vfill\vfill
\noindent \small Submitted to EJP on November 12, 2007, final version accepted on February 12, 2008.
\url{http://www.math.washington.edu/~ejpecp/viewarticle.php?id=1781}

\newpage
\section{Introduction}
\subsection{Model}
We analyze the asymptotic behavior of the sequence of random variables $\left(x(n,x_0)\right)_{n\in\N}$ defined by:
\begin{equation}\label{eqdefx}
\left\{\begin{array}{lcl}
x(0,x_0)&=&x_0\\
x_i(n+1,x_0)&=&\max_{j} \left(A_{ij}(n)+x_j(n,x_0)\right)
\end{array} \right.,
\end{equation}
where $\sAn$ is a stationary and ergodic sequence of random matrices with entries in \mbox{$\R\cup\{-\infty\}$}.
Moreover, we assume that $A(n)$  has at least one finite entry on each row, which is a necessary and sufficient condition for $x(n,x_0)$ to be finite. (Otherwise, some coefficients can be $-\infty$.)\\

Such sequences are best understood by introducing the so-called \mp algebra, which is actually a semiring.
\begin{defi}
The \mp semiring $\rma$ is the set $\R\cup\{-\infty\}$,  with  the max as a sum (i.e. 
$a \oplus b = \max(a, b)$) and the usual sum as a product (i.e. $a\otimes b = a + b$). In this semiring, the identity elements are $-\infty$ and~$0$.
\end{defi}
We also use the matrix and vector operations induced by the semiring structure.
For matrices $A, B$ with appropriate sizes, $(A\oplus B)_{ij} = A_{ij} \oplus B_{ij} = \max(A_{ij} , B_{ij})$, $(A\otimes B)_{ij} = 
\bigoplus_k A_{ik}\otimes B_{kj} = \max_k (A_{ik} + B_{kj} )$, and for a scalar $a\in\rma$, 
$(a\otimes  A)_{ij} = a\otimes  A_{ij} = a + A_{ij}$. 
Now, Equation~(\ref{eqdefx})  $x(n+1,x_0)\otimes A(n)x(n,x_0)$. In the sequel, all products of matrices by vectors or other matrices are to be understood in this structure.

For any integer $k\ge n$, we define the product of matrices $A(k,n):=A(k) \cdots A(n)$ with entries in this semiring. Therefore, we have $x(n,x_0)=A(n-1,0) x_0$ and if the sequence has indices in $\Z$, which is possible up to a change of probability space, we define a new random vector $y(n,x_0):=A(-1,-n)x_0$, which has the same distribution as $x(n,x_0)$.\\

Sequences defined by Equation~\ref{eqdefx} model a large class of discrete event dynamical systems. This class includes some models of operations research like timed event graphs (F.~Baccelli~\cite{Baccelli}), 1-bounded Petri nets (S.~Gaubert and J.~Mairesse~\cite{GaubertMairesseIEEE}) and some queuing networks (J. Mairesse~\cite{Mairesse}, B.~Heidergott~\cite{CaractMpQueuNet}) as well as many concrete applications. Let us cite job-shops models (G.~Cohen~et~al.\cite{cohen85a}), train networks (H.~Braker~\cite{Braker}, A.~de Kort and B.~Heidergott~\cite{RailwayMpKHA}), computer networks (F.~Baccelli and D.~Hong~\cite{TCPmp}) or a statistical mechanics model (R.~Griffiths~\cite{Griffiths}). For more details about modelling, see the books by F.~Baccelli and al.~\cite{BCOQ} and by B.~Heidergott and al.~\cite{MpAtWork}.

\subsection{Law of large numbers}
The sequences satisfying Equation~(\ref{eqdefx}) have been studied in many papers. If a matrix $A$ has at least one finite entry on each row, then $x\mapsto Ax$ is non-expanding for the $L^\infty$ norm. Therefore, we can assume that $x_0$ is the 0-vector, also denoted by~$0$, and we do it from now on.

We say that $\left(x(n,0)\right)_{n\in\N}$ defined in~(\ref{eqdefx}) satisfies the strong law of large numbers if $\left(\frac{1}{n}x(n,0)\right)_{n\in\N}$ converges almost surely. When it exists, the limit in the law of large numbers is called the \emph{cycle time} of $\sAn$ or $\sXn$, and may in principle be a random variable. Therefore, we say that $\sAn$ has a cycle time rather than $\sXn$ satisfies the strong law of large numbers.

Some sufficient conditions for the existence of this cycle time were given by J.E.~Cohen~\cite{Cohen}, F.~Baccelli and Liu~\cite{BaccelliLiu,Baccelli}, Hong~\cite{Hong} and more recently by Bousch and Mairesse~\cite{BouschMairesseEng}, the author~\cite{theseGM} or Heidergott et al.~\cite{MpAtWork}. 

Bousch and Mairesse proved (Cf.~\cite{BouschMairesseEng}) that, if $A(0)0$ is integrable, then the sequence $\left(\frac{1}{n}y(n,0)\right)_{n\in\N}$ converges almost-surely and in mean and that, under stronger integrability conditions, $\left(\frac{1}{n}x(n,0)\right)_{n\in\N}$ converges almost-surely if and only if the limit of $\left(\frac{1}{n}y(n,0)\right)_{n\in\N}$ is deterministic.
The previous results can be seen as providing sufficient conditions for this to happen.
Some results only assumed ergodicity of $\sAn$, some others independence. But, even in the i.i.d. case, it was still unknown, which sequences had a cycle time and which had none.\\

\junk{
Based on the part of the result in~\cite{Baccelli} that do not rely on the additional hypotheses of that article, we give a necessary condition for the limit of $\left(\frac{1}{n}y(n,x_0)\right)_{n\in\N}$ to be deterministic. This condition involves a graph defined by the support of law of $A(1)$. Moreover, whenever it exists, the limit of $\left(\frac{1}{n}x(n,0)\right)_{n\in\N}$ is given by Formula~(\ref{eqdefx}), which was proved in~\cite{Baccelli} under several additional assumptions but unknown when the convergence was proved by other means, like in~\cite{BouschMairesseEng, Theorem~6.7}. Those results are gathered in Theorem~\ref{thCN}.

Conversely, we show (Theorem~\ref{thiid1}) that the necessary condition is sufficient when the $A(n)$ are i.i.d. As a first step, we extend (Theorem~\ref{thHong}) a result of~\cite{Hong}. Then, we perform an induction, thanks to Proposition~\ref{thgene} and Lemma~\ref{lemMifini}.}

In this paper, we solve this long standing problem.
The main result  (Theorem~\ref{thiid1}) establishes a necessary and sufficient condition for the existence of the cycle time of $\sAn$.
Moreover, we show that this condition is necessary (Theorem~\ref{thCN}) but not sufficient (Example~\ref{exIslande}) when $\sAn$ is only ergodic or mixing. Theorem~\ref{thCN} also states that the cycle time is always  given by a formula (Formula~(\ref{eqcvgcecompx})), which was proved in Baccelli~\cite{Baccelli} under several additional conditions.\\

To state the necessary and sufficient condition, we extend the notion of graph of a random matrix from the \emph{fixed support} case, that is when the entries are either almost-surely finite or almost-surely equal to $-\infty$, to the general case. The analysis of its decomposition into strongly connected components allows us to define new submatrices, which must have almost-surely at least one finite entry on each row, for the cycle time to exist.

To prove the necessity of the condition, we use the convergence results of Bousch and Mairesse~\cite{BouschMairesseEng} and a result of Baccelli~\cite{Baccelli}. To prove the converse part of Theorem~\ref{thiid1}, we perform an induction on the number of strongly connected components of the graph. The first step of the induction (Theorem~\ref{thHong}) is an extension of a result of D.~Hong~\cite{Hong}.\\

The paper is organized as follows.
In Section~\ref{SecPresRes}, we state our results and give examples to show that the hypotheses are necessary. In Section~\ref{SecProofs}, we successively prove Theorem~\ref{thCN} and Theorem~\ref{thiid1}

\section{Results}\label{SecPresRes}
\subsection{Theorems}
In this section we attach a graph to our sequence of random matrices, in order to define the necessary condition and to split the problem for the inductive proof of the converse theorem.\\

Before defining the graph, we need the following result, which directly follows from Kingman's theorem and goes back to J.E.~Cohen~\cite{Cohen}:
\begin{th-def}[Maximal Lyapunov exponent]\label{defexptop}\ \\
If $\sAn$ is an ergodic sequence of random matrices with entries in $\rma$ such that the positive part of $\max_{ij} A_{ij}(0)$ is integrable, then the sequences $\left(\frac{1}{n} \max_i x_i(n,0)\right)_{n\in\N}$ and $\left(\frac{1}{n} \max_i y_i(n,0)\right)_{n\in\N}$ converge almost-surely to the same constant $\gamma\in\rma$, which is called the maximal (or top) Lyapunov exponent of $\left(A(n)\right)_{n\in\N}$.

We denote this constant by $\gamma\left(\left(A(n)\right)_{n\in\N}\right)$, or  $\gamma(A)$.
\end{th-def}
\begin{rems}\ 
\begin{enumerate}
\item The constant $\gamma(A)$ is well-defined even if $\sAn$ has a row without finite entry.
\item The variables $\max_i x_i(n,0)$ and  $\max_i y_i(n,0)$ are equal to \mbox{$\max_{ij} A(n-1,0)_{ij}$} and $\max_{ij} A(-1,-n)_{ij}$ respectively.
\end{enumerate}
\end{rems}

Let us define the graph attached to our sequence of random matrices as well as some subgraphs. We also set the notations for the rest of the text.
\begin{defi}[Graph of a random matrix]\label{defGA}
For every $x\in\rma^{[1,\cdots,d]}$ and every subset $I\subset [1,\cdots,d]$, we define the subvector $x^I:=(x_i)_{i\in I}.$

Let $\sAn$  be a stationary sequence of random matrices with values in $\rma^{d\times d}$. 
\begin{enumerate}[i)]
\item  The graph  of $\sAn$, denoted by $\G(A)$, is the directed graph whose nodes are the integers between 1 and d and whose arcs are the pairs $(i,j)$ such that $\P(A_{ij}(0)\neq-\infty)>0$.

\item To each  \emph{strongly connected component (s.c.c)} $c$ of $\G(A)$, we attach the submatrices $A^{(c)}(n):=(A_{ij}(n))_{i,j\in c}$ and the exponent $\gamma^{(c)}:=\gamma(A^{(c)})$.

Nodes which are not in a circuit are assumed to be alone in their s.c.c Those s.c.c are called trivial and they satisfy $A^{(c)}=-\infty ~a.s.$ and therefore $\gamma^{(c)}=-\infty$.

\item A s.c.c $\tilde{c}$ is \emph{reachable} from a s.c.c $c$ (resp. from a node $i$) if $c=\tilde{c}$ (resp. $i\in c$) or if there exists a path on $\G(A)$ from a node in $c$ (resp. from $i$) to a node in $\tilde{c}$. In this case, we write $c\rightarrow \tilde{c}$. (resp. $i\rightarrow \tilde{c})$.

\item To each s.c.c. $c$, we associate the set $\{c\}$ constructed as follows. First, one finds all s.c.c. downstream of $c$ with maximal Lyapunov exponent. Let $C$ be their union. Then the set $\{ c \}$ consists of all nodes  between $ c $  and $C$:
$$\{c\}:=\left\{i\in [1,d]\left|\exists \tilde{c}, c\rightarrow i\rightarrow \tilde{c}, \gamma^{(\tilde{c})}=\max_{c \rightarrow \bar{c} } \gamma^{(\overline{c})}\right.\right\}.$$
\end{enumerate}
\end{defi}

\begin{rem}[Paths on $\G(A)$]\label{remPaths}\ 
\begin{enumerate}
\item The products of matrices satisfy the following equation:
$$A(k,k-n)_{ij}=\max_{i_0=i,i_n=j}\sum_{l=0}^{n-1}A_{i_li_{l+1}}(k-l),$$
which can be read as '$A(k,k-n)_{ij}$ is the maximum of the weights of paths from $i$ to $j$ with length $n$ on $\G(A)$, the weight of the $l^\textrm{th}$ arc being given by $A(k-l)$'. For $k=-1$, it implies that $y_i(n,0)$  is the maximum of the weights of paths on $\G(A)$ with initial node~$i$ and length $n$ but $\gamma(A)$ is not really the maximal average weight of infinite paths, because the average is a limit and maximum is taken over finite paths, before the limit over~$n$. However, Theorem~\ref{decomplyap}, due to Baccelli and Liu~\cite{Baccelli,BaccelliLiu}, shows that the maximum and the limit can be exchanged.
\item Previous author used such a graph, in the fixed support case, that is when $P(A_{ij}(0)=-\infty)\in\{0,1\}$. In that case, the (random) weights where almost surely finite. Here, we can have weights equal to~$-\infty$, but only with probability strictly less than one.
\item In the literature, the isomorphic graph with weight $A_{ji}$ on arc $(i,j)$ is often used, although only in the fixed support case. This is natural in order to multiply vectors on their left an compute $x(n,0)$. Since we mainly work with $y(n,0)$ and thus multiply matrice on their right, our definition is more convenient.
\end{enumerate}
\end{rem}

With those definitions, we can state the announced necessary condition for $\left(x(n,X_0)\right)_{n\in\N}$ to satisfy a strong law of large numbers:
\begin{theo}\label{thCN}
Let $\left(A(n)\right)_{n\in\N}$ be a stationary and ergodic sequence of random matrices  with values in $\rma^{d\times d}$ and almost-surely at least one finite entry on each row, such that the positive part of $\max_{ij} A_{ij}(0)$ is integrable.
\begin{enumerate}
\item If the limit of $\left(\frac{1}{n}y(n,0)\right)_{n\in\N}$ is deterministic, then it is given by:
\begin{equation}\label{eqcvgcecompy}
\forall i\in [1,d], \lim_n\frac{1}{n}y_i(n,0)=\max_{i\rightarrow c}\gamma^{(c)}~\mathrm{ a.s.},
\end{equation}
That being the case, for every s.c.c $c$ of $\G(A)$, the submatrix $A^{\{c\}}$ of $A(0)$ whose indices are in $\{c\}$ almost-surely has at least one finite entry on each row.
\item If $\left(\frac{1}{n}x(n,0)\right)_{n\in\N}$ converges almost-surely, then its limit is deterministic and is equal to that of $\left(\frac{1}{n}y(n,0)\right)_{n\in\N}$, that is we have:
\begin{equation}\label{eqcvgcecompx}
\forall i\in [1,d], \lim_n\frac{1}{n}x_i(n,0)=\max_{i\rightarrow c}\gamma^{(c)}~\mathrm{ a.s.},
\end{equation}
\end{enumerate}
\end{theo}

To make the submatrices $A^{\{c\}}$ more concrete, we give on Fig.~\ref{FigGA} an example of a graph $\G(A)$ with the exponent~$\gamma^{(k)}$ attached to each s.c.c $c_k$ and we compute~$\{c_2\}$. The maximal Lyapunov exponent of s.c.c. downstream of~$c_2$, is~$\gamma^{(5)}$. The only s.c.c. downstream of~$c_2$ with this Lyapunov exponent is~$c_5$ and the only s.c.c. between $c_2$ and~$c_5$ is~$c_3$. Therefore, ${\{c_2\}}$ is the union of $c_2$, $c_3$ and $c_5$.

\begin{figure}[htbp]
\begin{center}
\caption{An example of computations on $\G(A)$}\label{FigGA}
\vspace*{0.2cm}
\begin{picture}(0,0)%
\includegraphics{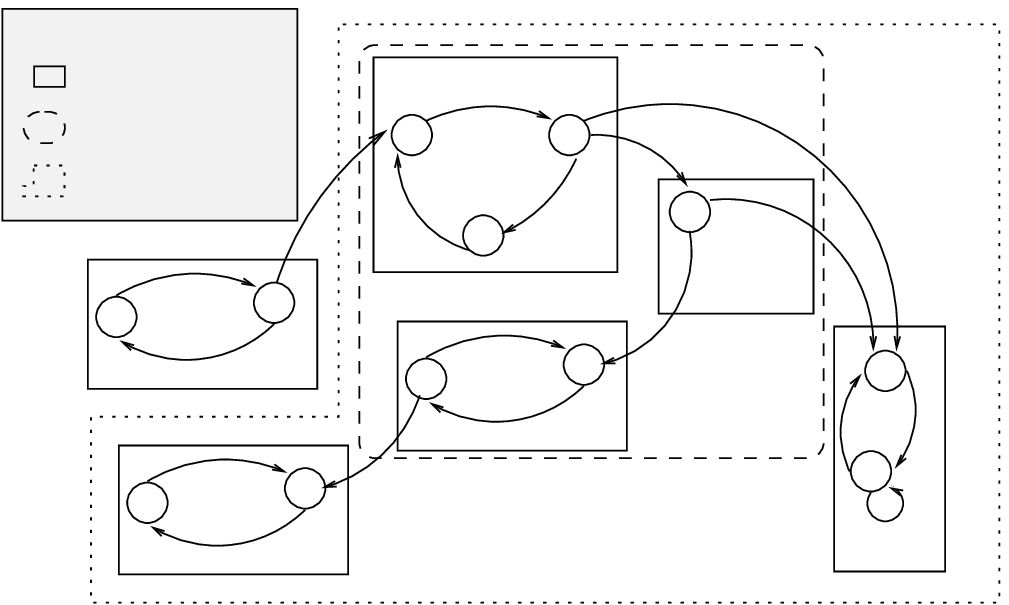}%
\end{picture}%
\setlength{\unitlength}{3947sp}%
\begingroup\makeatletter\ifx\SetFigFont\undefined%
\gdef\SetFigFont#1#2#3#4#5{%
  \reset@font\fontsize{#1}{#2pt}%
  \fontfamily{#3}\fontseries{#4}\fontshape{#5}%
  \selectfont}%
\fi\endgroup%
\begin{picture}(4809,2874)(-686,-4423)
\put(3578,-4194){\makebox(0,0)[b]{\smash{{\SetFigFont{8}{9.6}{\rmdefault}{\mddefault}{\updefault}{\color[rgb]{0,0,0}$\gamma^{(4)}=2$}%
}}}}
\put(557,-3339){\makebox(0,0)[b]{\smash{{\SetFigFont{8}{9.6}{\rmdefault}{\mddefault}{\updefault}{\color[rgb]{0,0,0}$\gamma^{(1)}=4$}%
}}}}
\put(706,-4231){\makebox(0,0)[b]{\smash{{\SetFigFont{8}{9.6}{\rmdefault}{\mddefault}{\updefault}{\color[rgb]{0,0,0}$\gamma^{(6)}=0$}%
}}}}
\put(2041,-3617){\makebox(0,0)[b]{\smash{{\SetFigFont{8}{9.6}{\rmdefault}{\mddefault}{\updefault}{\color[rgb]{0,0,0}$\gamma^{(5)}=3$}%
}}}}
\put(-524,-1711){\makebox(0,0)[lb]{\smash{{\SetFigFont{8}{9.6}{\rmdefault}{\mddefault}{\updefault}{\color[rgb]{0,0,0}Legend}%
}}}}
\put(-299,-1936){\makebox(0,0)[lb]{\smash{{\SetFigFont{8}{9.6}{\rmdefault}{\mddefault}{\updefault}{\color[rgb]{0,0,0}: $c$}%
}}}}
\put(-299,-2161){\makebox(0,0)[lb]{\smash{{\SetFigFont{8}{9.6}{\rmdefault}{\mddefault}{\updefault}{\color[rgb]{0,0,0}: \tiny ${\{c_2\}}$}%
}}}}
\put(-299,-2461){\makebox(0,0)[lb]{\smash{{\SetFigFont{8}{9.6}{\rmdefault}{\mddefault}{\updefault}{\color[rgb]{0,0,0}: $\bigcup_{c_2\rightarrow c}c$}%
}}}}
\put(1764,-2760){\makebox(0,0)[lb]{\smash{{\SetFigFont{8}{9.6}{\rmdefault}{\mddefault}{\updefault}{\color[rgb]{0,0,0}$\gamma^{(2)}=1$}%
}}}}
\put(2679,-2971){\makebox(0,0)[lb]{\smash{{\SetFigFont{8}{9.6}{\rmdefault}{\mddefault}{\updefault}{\color[rgb]{0,0,0}$\hspace*{-1mm}\gamma^{(3)}\hspace*{-1mm}=\hspace*{-1mm}-\infty$}%
}}}}
\end{picture}%

\end{center}
\end{figure}

The necessary and sufficient condition in the i.i.d. case reads
\begin{theo}[Independent case]\label{thiid1}
If $\left(A(n)\right)_{n\in\N}$ is a sequence of i.i.d. random matrices with values in $\rma^{d\times d}$ and almost-surely at least one finite entry on each row, such that $\max_{A_{ij}(0)\neq -\infty}|A_{ij}(0)|$ is integrable, then the sequence $\left(\frac{1}{n}x(n,0)\right)$ converges almost-surely
if and only if for every s.c.c $c$, the submatrix $A^{\{c\}}$ of $A(0)$ defined in Theorem~\ref{thCN} almost-surely has at least one finite entry on each row.
That being the case the limit is given by Equation~(\ref{eqcvgcecompx}).
\end{theo}

\begin{rem}\label{remThiid}
We also prove that, when $A(0)0\in\L^1$, the limit of  $\left(\frac{1}{n}y(n,0)\right)$ is deterministic if and only if the matrices $A^{\{c\}}$ almost-surely have at least one finite entry on each row.

The stronger integrability ensures the convergence of $\left(\frac{1}{n}x(n,0)\right)$ to this limit, like in~\cite[Theorem~6.18]{BouschMairesseEng}.
There, it appeared as the specialization of a general condition for uniformly topical operators, whereas in this paper it ensures that $B0$ is integrable for every submatrix $B$ of $A(0)$ with at least one finite entry on each row.

Actually, we prove that $\left(\frac{1}{n}x(n,0)\right)$ converges, provided that $\forall c, A^{\{c\}}0\in\L^1$, (see Proposition~\ref{thgene}). We chose to give a slightly stronger integrability condition, which is easier to check because it does not depend on $\G(A)$.
\end{rem}

\subsection{Examples}
To end this section, below are three examples that show that the independence is necessary but not sufficient to ensure the strong law of large numbers and that the integrability condition is necessary. We will denote by $x^\top$ the transpose of a vector~$x$.
\begin{ex}[Independence is necessary]\label{exIslande}
Let $A$ and $B$ be defined by
$$A=\left(\begin{array}{cc}1&-\infty\\-\infty&0\end{array}\right) 
\textrm{ and } 
B=\left(\begin{array}{cc}-\infty&0\\ 0&-\infty\end{array}\right).$$
For any positive numbers $\gamma_1$ and $\gamma_2$ such that $\gamma_1+\gamma_2<1$, we set $\delta=\frac{1-\gamma_1-\gamma_2}{2}$. Let $(A(n),i_n)_{n\in\N}$ be a stationnary version of the irreducible Markov chain on $\{A,B\}\times\{1,2\}$ with transition probabilities given by the diagram of Figure~\ref{FigCM}:
\begin{figure}[htbp]
\begin{center}
\caption{Transition probabilities of $\left(A(n),i_n\right)_{n\in\N}$}\label{FigCM}
\begin{VCPicture}{(0,2)(12,5)}
\tiny
\State[$A,1$]{(3,3)}{1}
\State[$B,2$]{(6,4)}{2}
\State[$A,2$]{(9,3)}{3}
\State[$B,1$]{(6,2)}{4}
\tiny
\ArcL{1}{2}{\delta}
\ArcL{2}{3}{$$\gamma_2$$}
\ArcL[0.5]{2}{4}{$$1-\gamma_2$$}
\ArcL{3}{4}{$$\delta$$}
\ArcL{4}{1}{$$\gamma_1$$}
\ArcL[0.5]{4}{2}{$$1-\gamma_1$$}
\LoopW{1}{$$1-\delta$$}
\LoopE{3}{$$1-\delta$$}
\end{VCPicture}
\end{center}
\end{figure}

Then, $\sAn$ is a strongly mixing sequence of matrices, which means that it satisfies
$$\E\left[f\left(A(0)\right)g\left(A(n)\right)\right]\rightarrow \E\left[f\left(A(0)\right)\right]\E\left[f\left(A(0)\right)\right]$$
 for any integrable functions $f$and $g$ on $\rma^{d\times d}$. Moreover, its support is the full shift $\{A,B\}^\N$, but we have
\begin{equation}\label{eqEx}
\P\left(\lim_n\frac{1}{n}y_1(n,0)=\gamma_1\right)= \gamma_1+\delta \textrm{ and } \P\left(\lim_n\frac{1}{n}y_1(n,0)=\gamma_2\right)= \gamma_2+\delta,
\end{equation}
and thus, according to Theorem~\ref{thCN}, $\left(\frac{1}{n}x(n,0)\right)_{n\in\N}$ does not converge.
Finally, even if $\sAn$ is a quickly mixing sequence, which means that it is in some sense close to i.i.d.\,, and $\G(A)$ is strongly connected, does $\sAn$ fail to have a cycle time.

To prove Equation~(\ref{eqEx}), let us denote by $\tau$ the permutation between $1$ and $2$ and by $g(C,i)$ the only finite entry on the $i^\mathrm{th}$ row of~$C$. It means that for any~$i$, $g(A,i)=A_{ii}$ and $g(B,i)=B_{i\tau(i)}$. Since all arcs of the diagram arriving to a node $(A,i)$ are coming from a node $(C,i)$, while those arriving at a node $(B,i)$ are coming from a node $(C,\tau(i))$, we almost surely have
$$x_{i_n}(n+1,0)-x_{i_{n-1}}(n,0)=g(A(n),i_n)) \textrm{ and } x_{\tau(i_n)}(n+1,0)-x_{\tau(i_{n-1})}(n,0)=g(A(n),\tau(i_n))),$$
and thus
\begin{eqnarray*}
x_{i_{n-1}}(n,0)=\sum_{k=0}^{n-1}g(A(k),i_k)&\textrm{ and }&
x_{\tau(i_{n-1})}(n,0)=\sum_{k=0}^{n-1}g(A(k),\tau(i_k)),\\
y_{i_{-1}}(n,0)=\sum_{k=1}^ng(A(-k),i_{-k})&\textrm{ and }&
y_{\tau(i_{-1})}(n,0)=\sum_{k=1}^ng(A(-k),\tau(i_{-k})).
\end{eqnarray*}

It is easily checked that the invariant distribution of the Markov chain is given by the following table:
\begin{center}
\begin{tabular}{|c||c|c|c|c|}
\hline
$x$		& $(A,1)$  & $(B,2)$  & $(A,2)$   & $(B,1)$ \\
\hline 
$\P((A(n),i_n)=x)$ & $\gamma_1$ & $\delta$ & $\gamma_2$ & $\delta$ \\
\hline 
\end{tabular}
\end{center}
and that $g$ is equal to $0$ except in $(A,1)$.

Therefore, we have 
\begin{eqnarray*}
\lim_n\frac{1}{n}y_{i_{-1}}(n,0)&=&\E\left(g(A(0),i_0)\right)=\P\left((A(0),i_0)=(A,1)\right)=\gamma_1\\
\lim_n\frac{1}{n}y_{\tau(i_{-1})}(n,0)&=&\E\left(g(A(0),\tau(i_0))\right)=\P\left((A(0),\tau(i_0))=(A,2)\right)=\gamma_2
\end{eqnarray*}

and consequently
$$\lim_n\frac{1}{n}y(n,0)=(\gamma_{i_{-1}}, \gamma_{\tau(i_{-1})})^\top a.s.$$
which implies Equation~(\ref{eqEx}).
\end{ex}

The next example, due to Bousch and Mairesse shows that the cycle time may not exist, even if the $A(n)$ are i.i.d.
\begin{ex}[Bousch and Mairesse, Independence is not sufficient]\label{exmairesse}
Let $\left(A(n)\right)_{n \in\N}$ be the  sequence of i.i.d. random variables taking values
$$B=\left(\begin{array}{ccc}0&-\infty&-\infty\\ 0&-\infty&-\infty\\ 0&1&1\end{array}\right) 
\textrm{ and } 
C=\left(\begin{array}{ccc}0&-\infty&-\infty\\ 0&-\infty&0\\ 0&0&-\infty\end{array}\right)$$
with probabilities $p>0$ and $1-p>0$.
Let us compute the action of $B$ and $C$ on vectors of type $(0,x,y)^\top$, with $x,y\ge0$:
$$B(0,x,y)^\top=(0,0,\max(x,y)+1)^\top\textrm{ and } C (0,x,y)^\top=(0,y,x)^\top.$$
Therefore $x_1(n,0)=0$ and $\max_ix_i(n+1,0)=\#\{0\le k\le n|A(k)=B\}$. In particular, if 
$A(n)=B$, then $x(n+1,0)=\left(0,0,\#\{0\le k\le n|A(k)=B\}\right)^\top$, and if $A(n)=C$ and $A(n-1)=B$, then $x(n+1,0)=\left(0,\#\{0\le k\le n|A(k)=B\},0\right)^\top$. Since $\left(\frac{1}{n}\#\{0\le k\le n|A(k)=B\}\right)_{n\in\N}$ converges almost-surely to $p$, we arrive at:
\begin{equation}\label{eqexmairesse}
\begin{array}{c}
\lim_n \frac{1}{n} x_1(n,0)=0~\mathrm{ a.s.}\\
\forall i\in\{2,3\}, \liminf_n \frac{1}{n} x_i(n,0)=0\textrm{ and }\limsup_n \frac{1}{n} x_i(n,0)=p\textrm{  a.s.}
\end{array}
\end{equation}
Therefore the sequence $\left(\frac{1}{n}x(n,0)\right)_{n\in\N}$ almost-surely does not converge.\\

We notice that $\G(A)$ has two s.c.c $c_1=\{1\}$ and $c_2=\{2,3\}$, with Lyapunov exponents $\gamma^{(c_1)}=0$ and $\gamma^{(c_2)}=p$, and $2\rightarrow 1$. Therefore, we check that the first row of $A^{\{c_2\}}$ has no finite entry with probability $p$.

\end{ex}

Theorem~\ref{thiid1} gives a necessary and sufficient condition for the existence of the cycle time of an i.i.d. sequence of matrices~$A(n)$ such that $\max_{A_{ij}(0)\neq -\infty}|A_{ij}(0)|$ is integrable.
But the limit of $\left(\frac{1}{n}y(n,0)\right)_{n\in\N}$ exists  as soon as $A(0)0$ is integrable. Thus, it would be natural to expect Theorem~\ref{thiid1} to hold under this weaker integrability assumption. However, it does not, as the example below shows.
\begin{ex}[Integrability]\label{exinteg}
Let $\left(X_n\right)_{n\in\Z}$ be an i.i.d. sequence of real variables satisfying $X_n\ge 1\textrm{ a.s.}$ and $\E(X_n)=+\infty.$ The sequence of matrices is defined by:
$$A(n)=\left(\begin{array}{ccc}
-X_n&-X_n&0\\
-\infty&0&0\\
-\infty&-\infty&-1
\end{array}\right)$$
A straightforward computation shows that $x(n,0)$ is $\left(\max(-X_n,-n),0,-n\right)^\top$ and $y(n,0)=\left(\max(-X_0,-n),0,-n\right)^\top$. It follows from Borel-Cantelli lemma that $\lim_n\frac{1}{n}X_n=0 \textrm{ a.s.}$ if and only if $\E(X_n)<\infty$. Hence $\left(\frac{1}{n}x(n,0)\right)_{n\in\N}$ converges to $(0,0,-1)^\top$ in probability but the convergence does not occur almost-surely.\\

Let us notice that the limit of $\left(\frac{1}{n}y(n,0)\right)_{n\in\N}$ is given by Remark~\ref{remThiid}: each s.c.c has exactly one node, $\gamma^{(1)}=-\E(X_n)=-\infty$, $\gamma^{(2)}=0$ and $\gamma^{(3)}=-1$.
\end{ex}

\section{Proofs}\label{SecProofs}
\subsection{Necessary conditions}
\subsubsection{Additional notations}
To interpret the results in terms of paths on $\G(A)$, and prove them, we redefine the $A^{\{c\}}$ and some intermediate submatrices.
\begin{defi}\label{defGiA}
To each s.c.c $c$, we attach three sets of elements.
\begin{enumerate}[i)]
\item Those that only depend on $c$ itself.
$$x^{(c)}(n,x_0):=A^{(c)}(n-1,0)(x_0)^{c}
 \textrm{ and } y^{(c)}(n,x_0):=A^{(c)}(-1,-n)(x_0)^{c} $$
\item Those that depend on the graph downstream of~$c$.
$$E_c:=\{\tilde{c}|c\rightarrow \tilde{c}\}
\textrm{, }
\gamma^{[c]}:=\max_{\tilde{c}\in E_c}\gamma^{(\tilde{c})}
\textrm{, }$$
$$F_c:=\bigcup_{\tilde{c}\in E_c}\tilde{c}
\textrm{, }
A^{[c]}(n):=\left(A_{ij}(n)\right)_{i,j\in F_c}$$
$$x^{[c]}(n,x_0):=A^{[c]}(n-1,0)(x_0)^{F_c} \textrm{ and }
y^{[c]}(n,x_0):=A^{[c]}(-1,-n)(x_0)^{F_c}.$$

\item Those that depend on ${\{c\}}$, as defined in Definition~\ref{defGA}.
$$G_c:=\{\tilde{c}\in E_c|\exists \hat{c}, c\rightarrow \tilde{c}\rightarrow \hat{c}, \gamma^{(\hat{c})}=\gamma^{[c]}\},$$
$$H_c:=\bigcup_{\tilde{c}\in G_c}\tilde{c}\textrm{ , }  A^{\{c\}}(n):=\left(A_{ij}(n)\right)_{i,j\in H_c} $$
$$x^{\{c\}}(n,x_0):=A^{\{c\}}(n-1,0)(x_0)^{H_c}\textrm{ and } y^{\{c\}}(n,x_0):=A^{\{c\}}(-1,-n)(x_0)^{H_c}.$$
\item A s.c.c $c$ is called \emph{dominating} if $G_c=\{c\}$, that is if for every $\tilde{c}\in E_c\backslash\{c\}$, we have: 
$\gamma^{(c)}>\gamma^{(\tilde{c})}.$
\end{enumerate}
\end{defi}

With those notations, the $\{c\}$ of Definition~\ref{defGA} is denoted by~$H_c$, while $A^{\{c\}}$ is~$A^{\{c\}}(0)$.

As in Remark~\ref{remPaths}, we notice that the coefficients $y^{(c)}_i(n,0)$, $y^{[c]}_i(n,0)$ and $y^{\{c\}}_i(n,0)$ are the maximum of the weights of paths on the subgraph of $\G(A)$ with nodes in $c$, $F_c$ and $H_c$ respectively.

Consequently $\gamma^{(c)}$, $\gamma(A^{[c]})$ and $\gamma(A^{\{c\}})$ are the maximal average weight of infinite paths on $c$, $F_c$ and $G_c$ respectively. Since $\gamma^{[c]}$ is the maximum of the $\gamma^{(\tilde{c})}$ for s.c.c $\tilde{c}$ downstream of~$c$, the interpretation suggests it might be equal to $\gamma(A^{[c]})$ and $\gamma(A^{\{c\}})$. That this is indeed true has been shown by F.~Baccelli~\cite{Baccelli}.

Clearly, $\gamma(A^{[c]})\ge \gamma(A^{\{c\}})\ge \gamma(A^{[c]})$, but the maximum is actually taken for finite paths, so that the converse inequalities are not obvious. 

\subsubsection{Formula for the limit}
 Up to a change of probability space, we can assume that $A(n)=A\circ\theta^n$, where $A$ is a random variable  and $(\Omega,\theta,\P)$ is an invertible ergodic  measurable dynamical system. We do it from now on.

Let $L$ be the limit of $\left(\frac{1}{n}y(n,0)\right)_{n\in\N}$, which exists according to~\cite[Theorem~6.7]{BouschMairesseEng} and is assumed to be deterministic.

By definition of $\G(A)$, if $(i,j)$ is an arc of $\G(A)$, then, with positive probability, we have $A_{ij}(-1)\neq -\infty$ and
 $$L_i=\lim_n \frac{1}{n}y_i(n,0)\ge \lim_n \frac{1}{n} (A_{ij}(-1)+y_j(n,0)\circ\theta^{-1})=0+ L_j\circ\theta^{-1}=L_j.$$

If $c \rightarrow \tilde{c}$, then for every $i\in c$ and $j\in \tilde{c}$, there exists a path on $\G(A)$ from $i$ to $j$, therefore $L_i\ge L_j$. Since this holds for every $j\in F_c$, we have:
\begin{equation}\label{eqlim}
L_i =\max_{j\in F_c} L_j
\end{equation}
To show that $\max_{j\in F_c} L_j=\gamma^{[c]}$, we have to study the Lyapunov exponents of sub-matrices.

The following proposition states some easy consequences of Definition~\ref{defGiA} which will be useful in the sequel.
\begin{prop}\label{propGiA}
The notations are those of Definition~\ref{defGiA}
\begin{enumerate}[i)]
\item For every s.c.c. $c$, $x^{[c]}(n,x_0)=x^{F_c}(n,x_0)$.
\item For every  s.c.c. $m$,  and every $i\in c$, we have:
$$x_i(n,0)= x_i^{[c]}(n,0)\ge x_i^{\{c\}}(n,0)\ge x_i^{(c)}(n,0).$$
\begin{equation}\label{OrdreDesy}
y_i(n,0)= y_i^{[c]}(n,0)\ge y_i^{\{c\}}(n,0)\ge y_i^{(c)}(n,0).
\end{equation}
\junk{$$\max_{i_l\in [1,d]}\sum_{l=0}^{n-1}A_{i_{l+1}i_l}(l) =\max_{i_l\in F_c}\sum_{l=0}^{n-1}A_{i_{l+1}i_l}(l) \ge \max_{i_l\in H_c}\sum_{l=0}^{n-1}A_{i_{l+1}i_l}(l) \ge \max_{i_l\in c}\sum_{l=0}^{n-1}A_{i_{l+1}i_l}(l).$$}
\item Relation $\rightarrow$ is a partial order, for both the nodes and the s.c.c.
\item If $A(0)$ has almost-surely at least one finite entry on each row, then for every s.c.c. $c$, $A^{[c]}(0)$ has almost-surely has least one finite entry on each row.
\item For every $\tilde{c}\in E_c$, we have $\gamma^{(\tilde{c})}\le\gamma^{[\tilde{c}]}\le\gamma^{[c]}$ and \mbox{$G_c=\{\tilde{c}\in E_c|\gamma^{[\tilde{c}]}=\gamma^{[c]}\}.$}
\end{enumerate} 
\end{prop}

The next result is about Lyapunov exponents. It is already in~\cite{Baccelli,BaccelliLiu} and its proof does not uses the additional hypotheses of those articles. For a point by point checking, see~\cite{theseGM}.
\begin{theo}[F.~Baccelli and Z.~Liu~\cite{Baccelli, BaccelliLiu, BCOQ}]\label{decomplyap}
If $\sAn$ is a stationary and ergodic sequence of random matrices with values in $\rma^{d\times d}$ such that the positive part of $\max_{i,j} A_{ij}$ is integrable, then $\gamma(A)=\max_{c}\gamma^{(c)}$.
\end{theo}

Applying this theorem to sequences $\left(A^{[c]}(n)\right)_{n\in\N}$ and $\left(A^{\{c\}}(n)\right)_{n\in\N}$, we get the following corollary.
\begin{coro}\label{propGiA2}
For every s.c.c. $c$, we have $$\gamma(A^{\{c\}})=\gamma(A^{[c]})=\gamma^{[c]}.$$
\end{coro}

It follows from Proposition~\ref{propGiA} and  the definition of Lyapunov exponents  that for every s.c.c $c$ of $\G(A)$,
$$\max_{i\in F_c}L_i=\lim_n\frac{1}{n}\max_{i\in F_c}y_i(n,0)=\gamma(A^{[c]}).$$

Combining this with Equation~(\ref{eqlim}) and Corollary~\ref{propGiA2}, we deduce that the limit of $\left(\frac{1}{n} y(n,0)\right)_{n\in\N}$ is given by Equation~(\ref{eqcvgcecompy}).

\subsubsection{\texorpdfstring{$A^{\{c\}}(0)$}{$A^{\{c\}}(0)$} has at least one finite entry on each row}
We still have to show that for every s.c.c $c$, $A^{\{c\}}(0)$ almost-surely has at least one finite entry on each row. Let us assume it has none. It means that there exists a s.c.c. $c$ and an $i\in c$ such that the set $$\{\forall j\in H_c, A_{ij}(-1)=-\infty\}$$ has positive probability. On this set, we have:
$$y_i(n,0)\le \max_{j\in F_c\backslash H_c} A_{ij}(-1) +\max_{j\in F_c\backslash H_c} y_j(n-1,0)\circ\theta^{-1}.$$
Dividing by $n$ and letting $n$ to $+\infty$, we have $L_i\le\max_{j\in F_c\backslash H_c} L_j$. Replacing $L$ according to Equation~(\ref{eqcvgcecompy}) we get $\gamma^{[c]}\le\max_{k\in E_c\backslash G_c} \gamma^{[k]}$. This last inequality contradicts Proposition~\ref{propGiA}~$v)$. Therefore, $A^{\{c\}}(0)$  has almost-surely at least one finite entry on each row.

\subsubsection{The limit is deterministic}
Let us assume that $\left(\frac{1}{n}x(n,0)\right)_{n\in\N}$  converges almost-surely to a limit $L'$.

It follows from~\cite[Theorem~6.7]{BouschMairesseEng} that $\left(\frac{1}{n}y(n,0)\right)_{n\in\N}$ converges almost-surely, thus we have
$$\frac{1}{n}y(n,0)-\frac{1}{n+1}y(n+1,0)\stackrel{\P}{\rightarrow} 0.$$
We compound each term of this relation by $\theta^{n+1}$ and, since $x(n,0)=y(n,0)\circ\theta^n $, it proves that:
$$\frac{1}{n}x(n,0)\circ\theta-\frac{1}{n+1}x(n+1,0)\stackrel{\P}{\rightarrow}0.$$
When $n$ tends to $+\infty$, it  becomes $L'\circ\theta -L'=0$. Since $\theta$ is ergodic, this implies that $L'$ is deterministic.\\

 Since $\frac{1}{n}y(n,0)=\frac{1}{n}x(n,0)\circ\theta^n$, $L'$ and $L$ have the same law. Since $L'$ is deterministic, $L=L'$ almost-surely, therefore $L$ is also the limit of $\left(\frac{1}{n}x(n,0)\right)_{n\in\N}$.
This proves formula~(\ref{eqcvgcecompx}) and concludes the proof of Theorem~\ref{thCN}

\subsection{Sufficient conditions}
\subsubsection{Right products}
In this section, we prove the following proposition, which is a converse to Theorem~\ref{thCN}. In the sequel, $\1$~will denote the vector all coordinates of which are equal to~$1$.
\begin{prop}\label{thgene}
Let $\left(A(n)\right)_{n\in\N}$ be an ergodic sequence of random matrices with values in $\rma^{d\times d}$ such that the positive part of $\max_{ij} A_{ij}(0)$ is integrable and that the three following hypotheses are  satisfied:
\begin{enumerate}
\item For every s.c.c $c$ of $\G(A)$, $A^{\{c\}}(0)$ almost-surely has at least one finite entry on each row.
\item For every dominating s.c.c $c$ of $\G(A)$, $\lim_n\frac{1}{n}y^{(c)}(n,0)=\gamma^{(c)}\1~\mathrm{ a.s.}$
\item For every subsets $I$ and $J$  of $[1,\cdots,d]$, such that random matrices $\tilde{A}(n)=\left(A_{ij}(n)\right)_{i,j\in I\cup J}$ almost-surely have at least one finite entry on each row and split along $I$ and $J$ following the equation
\begin{equation}\label{decompblocs}
\tilde{A}(n)=:\left(\begin{array}{cc}
B(n)&D(n)\\
-\infty&C(n)
\end{array}\right),
\end{equation}
such that $\G(B)$ is strongly connected and  $D(n)$ is not almost-surely $(-\infty)^{I\times J}$, we have:
\begin{equation}\label{eqexistchem}
\P\left(\left\{\exists i\in I, \forall n\in\N, \left(B(-1)\cdots B(-n)D(-n-1)0\right)_i=-\infty\right\}\right)=0.
\end{equation}
\end{enumerate}
Then the limit of $\left(\frac{1}{n}y(n,0)\right)_{n\in\N}$ is given by Equation~(\ref{eqcvgcecompy}).\\

If Hypothesis~1. is strengthened by demanding that $A^{\{c\}}(0)0$ is integrable, then the sequence $\left(\frac{1}{n}x(n,0)\right)_{n\in\N}$ converges almost-surely and its limit is given by Equation~(\ref{eqcvgcecompx}).
\end{prop}
Hypothesis~1. is necessary according to Theorem~\ref{thCN}, Hypothesis~2 ensures the basis of the inductive proof, while Hypothesis~3 ensures the inductive step.
\begin{rem}[Non independent case]
Proposition~\ref{thgene} does not assume the independence of the $A(n)$.
Actually, it also implies that $\left(\frac{1}{n}x(n,0)\right)_{n\in\N}$ almost surely if the $A(n)$ have fixed support (that is $\P(A_{ij}(n)=-\infty)\in\{0,1\}$) and the powers of the shift are ergodic, which is an improvement of~\cite{Baccelli}. It also allows to prove the convergence when the diagonal entries of the $A(n)$ are almost surely finite, under weaker integrability conditions than in~\cite{BouschMairesseEng} (see~\cite{LGNGM} or~\cite{theseGM} for details).
\end{rem}

\begin{rem}[Paths on $\G(A)$, continued]
Let us interpret the three hypotheses with the paths on~$\G(A)$.
\begin{enumerate}
\item The hypothesis on $A^{\{c\}}(0)$ means that, whatever the initial condition~\mbox{$i\in c$,} there is always an infinite path beginning in~$i$ and not leaving~$H_c$.
\item The hypothesis on dominating s.c.c means that, whatever the initial condition~$i$ in a dominating s.c.c~$c$, there is always a path beginning in~$i$ with average weight $\gamma^{(c)}$. 
The proof of Theorem~\ref{decomplyap} (see~\cite{Baccelli} or~\cite{theseGM}) can be adapted to show that it is a necessary condition.
\item We will use the last hypothesis with $\tilde{A}(n)=A^{\{c\}}(n)$, $B(n)=A^{(c)}(n)$.  It means that there is a path from~\mbox{$i\in c$}, to $H_c\backslash c$.
Once we know that the limit of $\left(\frac{1}{n}y(n,0)\right)_{n\in\N}$ is given by Equation~(\ref{eqcvgcecompy}) this hypothesis is obviously necessary when $\gamma^{(c)}<\gamma^{[c]}$.
\end{enumerate}
\end{rem}

The remainder of this subsection is devoted to the proof of Proposition~\ref{thgene}.
It follows from Propositions~\ref{propGiA} and~\ref{propGiA2} and the definition of Lyapunov exponents that we have, for every s.c.c $c$ of $\G(A)$,
\begin{equation}\label{eqmajor}
\limsup_n\frac{1}{n}y^{c}(n,0)\le\gamma^{[c]}\1~\mathrm{a.s.}
\end{equation}

Therefore, it is sufficient to show that $\liminf_n\frac{1}{n}y^{c}(n,0)\ge\gamma^{[c]}\1~\mathrm{a.s.}$ Because of Proposition~\ref{propGiA}~$i)$, 
\begin{equation}\label{eqrec}
\lim_n\frac{1}{n}y^{\{c\}}(n,0)=\gamma^{[c]}\1.
\end{equation}
is a stronger statement. We prove Equation~(\ref{eqrec}) by induction on the size of $G_c$. The initialization of the induction is exactly Hypothesis~$2.$ of Proposition~\ref{thgene}.

Let us assume that Equation~(\ref{eqrec}) is satisfied by every $c$ such that the size of $G_c$ is less than $N$, and let $c$ be such that the size of $G_c$ is~$N+1$. Let us take $I=c$ and $J=H_c\backslash c$. If $c$ is not trivial, it is  the situation of Hypothesis~$3.$ with $\tilde{A}=A^{\{c\}}$, which almost-surely has at least one finite entry on each row thanks to Hypothesis~$1.$ Therefore,  Equation~(\ref{eqexistchem}) is satisfied. If $c$ is trivial, $\G(B)$ is not strongly connected, but Equation~(\ref{eqexistchem}) is still satisfied because $D(-1)0=(\tilde{A}(-1)0)^I\in\R^I$.

Moreover, $J$ is the union of the $\tilde{c}$ such that $\tilde{c}\in G_c\backslash\{c\}$, thus the  induction hypothesis implies that:
$$\forall j\in J, j\in \tilde{c}\Rightarrow\lim_n\frac{1}{n}\left(C(-1,-n)0\right)_j=\lim_n\frac{1}{n}y^{\{\tilde{c}\}}_j(n,0)=\gamma^{[\tilde{c}]}~\mathrm{ a.s. }. $$
Because of Corollary~\ref{propGiA2}~$ii)$, $\gamma^{[\tilde{c}]}=\gamma^{[c]},$ therefore the right side of the last equation is $\gamma^{[c]}$ and we have:
\begin{equation}\label{eqrecJ}
\lim_n\frac{1}{n}(y^{\{c\}})^J(n,0)=\lim_n \frac{1}{n}C(-1,-n)0= \gamma^{[c]}\1~\mathrm{ a.s. }.
\end{equation}

Equation~(\ref{eqexistchem}) ensures that, for every $i\in I$, there exists almost-surely a $T\in \N$ and a $j\in J$ such that $\left(B(-1,-T)D(-T-1)\right)_{ij}\neq-\infty$. Since we have $\lim_n\frac{1}{n}\left(C(-T,-n)0\right)_{j}=\gamma^{[c]} \textrm{ a.s.}$, it implies that:
\begin{eqnarray*}
\lefteqn{\liminf_n\frac{1}{n} y_i^{\{c\}}(n,0)}\\
&\ge& \lim_n \frac{1}{n}\left(B(-1,-T)D(-T-1)\right)_{ij}+ \lim_n\frac{1}{n}\left(C(-T,-n)0\right)_{j}=\gamma^{[c]}~\mathrm{ a.s. }
\end{eqnarray*}
Because of upper bound~(\ref{eqmajor}) and inequality~(\ref{OrdreDesy}), it implies that
$$\lim_n\frac{1}{n}(y^{\{c\}})^I(n,0)= \gamma^{[c]}\1~\mathrm{ a.s. }.$$
which, because of Equation~(\ref{eqrecJ}), proves Equation~(\ref{eqrec}). This  concludes the induction and  the proof of Proposition~\ref{thgene}.

\subsubsection{Left products}
As recalled in the introduction, T.~Bousch an J.~Mairesse proved that $\left(\frac{1}{n}x(n,0)\right)_{n\in\N}$ converges almost-surely as soon as the limit of $\left(\frac{1}{n}y(n,0)\right)_{n\in\N}$ is deterministic. Therefore, the hypotheses of Proposition~\ref{thgene} should imply the existence of the cycle time. But the theorem in~\cite[Theorem~6.18]{BouschMairesseEng} assumes a reinforced integrability assumption, that is not necessary for our proof. We will prove the following proposition in this section:
\begin{prop}\label{thgeneGche}
Let $\left(A(n)\right)_{n\in\N}$ be an ergodic sequence of random matrices with values in $\rma^{d\times d}$  such that the positive part of $\max_{ij} A_{ij}(0)$ is integrable and that satisfies the three hypotheses of Proposition~\ref{thgene}.

If Hypothesis~1. is strengthened by demanding that $A^{\{c\}}(0)0$ is integrable, then the sequence $\left(\frac{1}{n}x(n,0)\right)_{n\in\N}$ converges almost-surely and its limit is given by Equation~(\ref{eqcvgcecompx}).
\end{prop}

To deduce the results on $x(n,0)$ from those on $y(n,0)$, we introduce the following theorem-definition, which is a special case of J.-M.~Vincent~\cite[Theorem 1]{vincent} and  directly follows from Kingman's theorem:

\begin{th-def}[J.-M.~Vincent~\cite{vincent}]\label{thvincent}
If $\left(A(n)\right)_{n\in\Z}$ is a stationary and ergodic sequence of random matrices  with values in $\rma^{d\times d}$ and almost-surely at least one finite entry on each row such that $A(0)0$ is integrable, then there are two real numbers $\gamma(A)$ and $\gamma_b(A)$ such that 
$$\lim_n\frac{1}{n}\max_ix_i(n,0)=\frac{1}{n}\max_iy_i(n,0)=\gamma(A) ~\mathrm{ a.s.} $$
$$\lim_n\frac{1}{n}\min_ix_i(n,0)=\frac{1}{n}\min_iy_i(n,0)=\gamma_b(A) ~\mathrm{ a.s.} $$
\end{th-def}
It implies the following corollary, which makes the link between the results on~$\left(y(n,0)\right)_{n\in\N}$ and those on~$\left(x(n,0)\right)_{n\in\N}$ when all $\gamma^{[c]}$ are equal, that is when $\gamma(A)=\gamma_b(A)$.
\begin{coro}\label{lempassgauche}
If $\left(A(n)\right)_{n\in\Z}$ is a stationary and ergodic sequence of random matrices  with values in $\rma^{d\times d}$ and almost-surely at least one finite entry on each row such that $A(0)0$ is integrable then $$ \lim_n\frac{1}{n}x(n,0)=\gamma(A)\1\textrm{ if and only if }\lim_n\frac{1}{n}y(n,0)=\gamma(A)\1.$$
\end{coro}

Let us go back to the proof of the general result on $\left(x(n,0)\right)_{n\in\N}$.
Because of Proposition~\ref{propGiA} and Proposition~\ref{propGiA2} and  the definition of Lyapunov exponents, we already have, for every s.c.c $c$ of $\G(A)$,
$$\limsup_n\frac{1}{n}x^{c}(n,0)\le\gamma^{[c]}\1~\mathrm{a.s.}$$

Therefore it is sufficient to show that $\liminf_n\frac{1}{n}x^{c}(n,0)\ge\gamma^{[c]}\1~\mathrm{a.s.}$ and even that
$$\lim_n\frac{1}{n}x^{\{c\}}(n,0)=\gamma^{[c]}\1.$$
Because of corollary~\ref{lempassgauche}, it is equivalent to $\lim_n\frac{1}{n}y^{\{c\}}(n,0)=\gamma^{[c]}\1.$
Since all s.c.c of $\G(A^{\{c\}})$ are s.c.c of $\G(A)$ and have the same Lyapunov exponent~$\gamma^{(c)}$, it follows from the result on the~$y(n,0)$ applied to $A^{\{c\}}$.

\subsection{Independent case}
In this section, we prove Theorem~\ref{thiid1}. 

Because of Theorem~\ref{thCN}, it is sufficient to show that, if, for every s.c.c $c$, $A^{\{c\}}$ almost-surely has at least one finite entry on each row, then the sequence $\left(\frac{1}{n}x(n,0)\right)$ converges almost-surely. To do this, we will prove that, in this situation, the hypotheses of Proposition~\ref{thgeneGche} are satisfied. Hypothesis~$1.$ is exactly Hypothesis~$1.$ of Theorem~\ref{thiid1} and Hypotheses $2.$ and $3.$ respectively follow from the next lemma and theorem.

\begin{defi} 
For every matrix $A\in\rma^{d\times d}$, the pattern matrix $\widehat{A}$ is defined by $\widehat{A}_{ij}=-\infty$ if $A_{ij}=-\infty$ and $A_{ij}=0$ otherwise.
\end{defi}
For every matrix $A,B\in\rma^{d\times d}$, we have $\widehat{A  B}=\widehat{A} \widehat{B}$.

\begin{lem}\label{lemMifini}\ 
Let $\sAn$ be a stationary sequence of random matrices with values in $\rma^{d\times d}$ and almost-surely at least one finite entry on each row. Let us assume that there exists a partition $(I,J)$ of $[1,\cdots,d]$ such that $A=\tilde{A}$ satisfy Equation~(\ref{decompblocs}), with $\G(B)$ strongly connected.
For every $i\in I$, let us define
$$\A_i:=\left\{\forall n\in\N, \left(B(1,n)D(n+1)0\right)_i=-\infty\right\}.$$
\begin{enumerate}
\item If $\omega\in\A_i$, then we have $\forall n\in\N, \exists i_n\in I \left(B(1,n)\right)_{ii_n}\neq -\infty.$
\item If the set $\mE=\left\{M\in \{0,-\infty\}^{d\times d}\left|\P\left(\widehat{A}(1,n)=M\right)>0\right.\right\}$ is a semigroup, and if $\P\left(D=(-\infty)^{I\times J}\right)<1$, then for every $i\in I$, we have $ \P(\A_i)=0.$
\end{enumerate}
\end{lem}

\begin{proof}\ 
\begin{enumerate}
\item For every $\omega\in\A_i$, we prove our result by induction on $n$.

Since the $A(n)$ almost-surely have at least one finite entry on each row, there exists an $i_1\in [1,\cdots,d]$, such that $A_{ii_1}(1)\neq -\infty$.
Since $\left(D(1)0\right)_{i}=-\infty$, every entry on row~$i$ of~$D(1)$ is $-\infty$, that is $A_{ij}(1)=-\infty$ for every $j\in J$, therefore $i_{1}\in I$ and $B_{ii_1}(1)=A_{ii_1}(1)\neq -\infty$.

Let us assume that the sequence is defined up to rank $n$.
Since $A(n+1)$ almost-surely has at least one finite entry on each row, there exists an $i_{n+1}\in [1,\cdots,d]$, such that $A_{i_ni_{n+1}}(n+1)\neq -\infty$. 

Since $\omega\in\A_i$, we have:
$$-\infty=\left(B(1,n)D(n+1)0\right)_i\ge \left(B(1,n)\right)_{ii_n}+\left(D(n+1)0\right)_{i_{n}},$$
therefore $\left(D(n+1)0\right)_{i_{n}}=-\infty$.

It means that every entry on  row $i_n$ of $D(n+1)$ is $-\infty$, that is \mbox{$A_{i_nj}(n+1)=-\infty$} for every $j\in J$, therefore $i_{n+1}\in I$ and
$$B_{i_ni_{n+1}}(n+1)=A_{i_ni_{n+1}}(n+1)\neq -\infty.$$

Finally, we have:
$$\left(B(1,n+1)\right)_{ii_{n+1}}\ge\left(B(1,n)\right)_{ii_{n}}+B_{i_ni_{n+1}}(n+1)\neq -\infty .$$

\item 
As a first step, we want to construct a matrix $M\in\mE$ such that
$$\forall i\in I, \exists j\in J, M_{ij}=0.$$

Since $\P\left(D=(-\infty)^{I\times J}\right)<1$, there are $\alpha\in I$, $\beta\in J$ and $M^0\in\mE$ with $M^0_{\alpha\beta}=0$. 
For any~$i\in I$, since $\G(B)$ is strongly connected, there is $M\in\mE$ such that $M\in\mE$ and $M_{i\alpha}=0$. Therefore $M^i=MM^0$ is in $\mE$ and satisfies $M^i_{i\beta}=0$.

Now let us assume $I=\{\alpha_1,\cdots,\alpha_m\}$ and define by induction the finite sequence of matrices $P^k$.
\begin{itemize}
 \item $P^1=M^{\alpha_1}$
 \item If there exists $j\in J$ such that $P^{k}_{\alpha_{k+1}j}=0$, then $P^{k+1}=P^k$.
Else, since the matrices have at least one finite entry on each row, there is an $i\in I$, such that $P^{k}_{\alpha_ki}$, and $P^{k+1}=P^kM^i$.
\end{itemize}
It is easily checked that such $P^k$ satisfy, $$\forall l\le k, \exists j\in J, P^k_{\alpha_lj}=0.$$
Therefore, we set $M=P^m$ and denote by $p$ the smallest integer such that $\P\left(\widehat{A}(1,p)=M\right)>0$\\

Now, it follows from the definition of $\mE$ and the ergodicity of $\sAn$ that there is almost surely an $N\in N$, such that $\widehat{A}(N+1,N+p)=M$.

On $\A_i$, that would define a random $j_N\in J$ such that $M_{i_Nj_N}=0$, where $i_N$ is defined according to the first point of the lemma. Then, we would have 
$$(A(1,N+p))_{ij_N}\ge(A(1,N))_{ii_N}+(A(N+1,N+p))_{i_Nj_N}>-\infty$$
But  $\A_i$ is defined as the event on which there is never a path from $i$ to $J$, so that we should have $\forall n\in\N,\forall j\in J, A(1,n))_{ij}=-\infty.$

Finally, $\A_i$ is included in the negligible set $\left\{\forall n\in\N, \widehat{A}(n+1,n+p)\neq M\right\}$.
\end{enumerate} 
\end{proof}

\begin{theo}\label{thHong}
If $\left(A(n)\right)_{n\in\N}$ is a sequence of i.i.d. random matrices with values in $\rma^{d\times d}$ such that the positive part of $\max_{i j}A_{ij}(0)$ is integrable, $A(0)$ almost-surely has at least one finite entry on each row and $\G(A)$ is strongly connected, then we have  $$\forall i\in [1,d], \lim_n\frac{1}{n} y_i(n,0)=\gamma(A)$$.
\end{theo}
This theorem is stated by D.~Hong in the unpublished~\cite{Hong}, but the proof is rather difficult to understand and it is unclear if it holds when $A(1)$ takes infinitely many values. Building on~\cite{BouschMairesseEng}, we now give a short proof of this result.
\begin{proof}
According to~\cite[Theorem~6.7]{BouschMairesseEng}, $\left(\frac{1}{n}y(n,0)\right)_{n\in\N}$ converges a.s. We have to show that its limit is deterministic.

The sequence $R(n):=\widehat{A}(-1,-n)$ is a Markov chain with  states space is $$\left\{M\in\{0,-\infty\}^{d\times d}\left|M0=0\right.\right\}$$ and whose transitions are defined by: $$\P\left(R(n+1)=F|R(n)=E\right)=\P\left(\widehat{E  A(1)}=F\right).$$
For every $i,j\in I$, we have $R_{ij}(n)=0$ if and only if $\left(A(-1,-n)\right)_{ij}\neq-\infty$.

Let $i$ be any integer in $\{1,\cdots,d\}$ and $E$ be a recurrent state of $\left(R(n)\right)_{n\in\N}$. There exists a $j\in [1,\cdots,d]$ such that $E_{ij}=0$. Since $\G(A)$ is strongly connected, there exists a $p\in\N$, such that $\left(B(-1,-p)\right)_{ji}\neq-\infty$ with positive probability. Let $G$ be such that $\P\left(\left(B(-1,-p)\right)_{ji}\neq-\infty, \widehat{B}(-1,-p)=G\right)>0$. Now, $F=EG$ is a state of the chain, reachable from state $E$ and such that $F_{ii}=0$. Since $E$ is recurrent, so is $F$ and $E$ and $F$ belong to the same recurrence class.

Let $\mathcal{E}$ be a set with exactly one matrix $F$ in each recurrence class, such that $F_{ii}=0$.
Let $S_n$ be the $n^\mathrm{th}$ time $\left(R(m)\right)_{m\in\N}$ is in $\mathcal{E}$.

Since the Markov chain has finitely many states and $\mathcal{E}$ intersects every recurrence class, $S_n$ is almost-surely finite, and even integrable. Moreover, the $S_{n+1}-S_n$ are i.i.d. (we set $S_0=0$) and so are the $A(-S_n-1,-S_{n+1})$. Since $P(S_1>k)$ decreases exponentially fast, $A(-1,-S_1)0$ is integrable and thus the sequence $\left(\frac{1}{n}y(S_n,0)\right)_{n\in\N}$  converges a.s. Let us denote its limit by~$l$. 

Let us denote by $\mathcal{F}_0$ the $\sigma$-algebra generated by the random matrices $A(-S_n-1,-S_{n+1})$. Then $l$ is $\mathcal{F}_0$ measurable, and the independence of the $A(-S_n-1,-S_{n+1})$ means that $(\Omega,\mathcal{F}_0,\P,\theta^{S_1})$ is an ergodic measurable dynamical system. Because of the choice of $S_1$, we have $l_i\ge l_i\circ\theta^{S_1}$, so that $l_i$ is deterministic.

Now, let us notice that the limit of $\frac{1}{n}y_i(n,0)$ is that of $\frac{1}{S_n}y_i(S_n,0)$, that is $\frac{l_i}{\E(S_1)}$, which is deterministic.

This means that $\lim\frac{1}{n}y_i(n,0)$ is deterministic for any~$i$, and, according to Theorem~\ref{thCN}, it implies that it is equal to $\gamma(A)$.
\end{proof}

\section{Acknowledgements}
This article is based on my work during my PhD at Universit\'e de Rennes~1, as a JSPS postdoctoral fellow at Keio University, and as ATER at Universit\'e Paris-Dauphine. During this time, many discussions with Jean Mairesse have been a great help. This paper owes much to him. I am also grateful to the anonymous reviewer for valuable suggestions of improvements in the presentation.

\def\cprime{$'$}

\end{document}